\documentclass[12pt]{amsart}
\usepackage{amsthm}
\usepackage{amssymb}
\usepackage{cite}
\usepackage{longtable}
\usepackage{geometry}
\usepackage{enumerate}
\usepackage{setspace}
\usepackage{tikz-cd}
\usetikzlibrary{matrix, arrows}
\usepackage[unicode=true]
 {hyperref}
\hypersetup{
colorlinks=true,
urlcolor=black,
citecolor=blue,
linkcolor=blue,
}

\allowdisplaybreaks

\newtheorem{thm}{Theorem}[section]
\newtheorem{prop}[thm]{Proposition}
\newtheorem{lem}[thm]{Lemma}
\newtheorem{cor}[thm]{Corollary}

\theoremstyle{definition}

\newtheorem*{assumption}{Assumptions}

\theoremstyle{remark}
\newtheorem{remark}[thm]{Remark}

\numberwithin{equation}{section}

\newcommand{\ff}{\mathfrak{f}}
\newcommand{\fg}{\mathfrak{g}}

\newcommand{\Aut}{\mathrm{Aut}}
\newcommand{\Hol}{\mathrm{Hol}}
\newcommand{\Hom}{\mathrm{Hom}}

\newcommand{\conj}{\mathrm{conj}}
\newcommand{\pmmod}{\hspace{-3mm}\pmod}

\begin{document}

\large 

\title[Hopf-Galois structures and skew braces]{Hopf-Galois structures on cyclic extensions \\and skew braces with cyclic multiplicative group}
\author{Cindy (Sin Yi) Tsang}
\address{Department of Mathematics, Ochanomizu University, 2-1-1 Otsuka, Bunkyo-ku, Tokyo, Japan}
\email{tsang.sin.yi@ocha.ac.jp}
\urladdr{http://sites.google.com/site/cindysinyitsang/} 

\date{\today}

\maketitle

\vspace{-5mm}

\begin{abstract}
Let $G$ and $N$ be two finite groups of the same order. It is well-known that the existences of the following are equivalent.
\begin{enumerate}[$(a)$]
\item a Hopf-Galois structure of type $N$ on any Galois $G$-extension
\item a skew brace with additive group $N$ and multiplicative group $G$
\item a regular subgroup isomorphic to $G$ in the holomorph of $N$ 
\end{enumerate}
We shall say that $(G,N)$ is \emph{realizable} when any of the above exists. Fixing $N$ to be a cyclic group, W. Rump (2019) has determined the groups $G$ for which $(G,N)$ is realizable. In this paper, fixing $G$ to be a cyclic group instead, we shall give a complete characterization of the groups $N$ for which $(G,N)$ is realizable.
\end{abstract}

\vspace{-1mm}

\tableofcontents

\vspace{-6mm}

\section{Introduction}

Let $G$ and $N$ be two finite groups of the same order. It is known that the existences of the following are equivalent (see \cite[Chapter 2]{Childs book} and \cite{Skew braces}).
\begin{enumerate}[$(a)$]
\item a Hopf-Galois structure of type $N$ on any Galois $G$-extension
\item a skew brace with additive group $N$ and multiplicative group $G$
\item a regular subgroup isomorphic to $G$ in the holomorph of $N$ 
\end{enumerate}
Here, the \emph{holomorph} of $N$ is defined to be
\[ \mathrm{Hol}(N) =\lambda(N)\rtimes \Aut(N)= \rho(N)\rtimes \Aut(N),\]
where $\lambda$ and $\rho$ denote the left and right regular representations
\[ \lambda(\eta) = ( x\mapsto \eta x), \,\ \rho(\eta) = (x\mapsto x\eta^{-1})\,\ \mbox{ for } \eta,x\in N,\]
and a subgroup $\mathcal{G}$ of $\Hol(N)$ is said to be \emph{regular} if $\mathcal{G}\longrightarrow N;\hspace{1mm}\sigma\mapsto \sigma(1_N)$ is a bijection. Following \cite{biskew}, we shall say that $(G,N)$ is \emph{realizable} when any of the above conditions is satisfied. Let us note that skew braces are ring-like structures introduced to study solutions to the Yang-Baxter equation.

\vspace{1.5mm}

Notice that $\lambda(N),\rho(N)\simeq N$ are regular subgroups of $\Hol(N)$, so the pair $(G,N)$ is realizable when $G\simeq N$. But whether $(G,N)$ is realizable depends upon the groups $G$ and $N$ when $G\not\simeq N$. It is therefore natural to ask which pairs $(G,N)$ are realizable. For example, when $G$ is fixed to be
\begin{itemize}
\item any group of squarefree order \cite{AB,AB', AB''},
\item any group of order $p^3$ with $p$ a prime \cite{NZ},
\item any non-abelian simple and more generally quasisimple group \cite{Byott simple,Tsang QS},
\item the symmetric group $S_n$ with $n\geq 5$ \cite{Tsang Sn},
\item the automorphism group of any sporadic simple group \cite{Tsang ASG},
\end{itemize} 
the groups $N$ for which the pair $(G,N)$ is realizable are completely known. There are also other papers (for example \cite{Byott soluble, FCC, Bachiller, Tsang solvable, Nasy}) which investigate necessary relations between $G$ and $N$ in order for $(G,N)$ to be realizable.

\vspace{1.5mm}

Cyclic groups have the simplest structure of all groups. It is then natural to ask for which groups $N$ is the pair $(C_n,N)$ realizable for the cyclic group $C_n$ of order $n$. The purpose of this paper is to characterize all such $N$. 

\vspace{1.5mm}

Let us first recall some known results. For $n$ an odd prime power, we have:

\begin{prop}\label{p prop1}Let $N$ be a group of order $p^m$ with $p$ an odd prime. Then the pair $(C_{p^m},N)$ is realizable if and only if $N\simeq C_{p^m}$.
\end{prop}
\begin{proof}See \cite[Theorem 4.5]{Kohl} or alternatively \cite[Theorem 1.5]{Tsang HG}.
\end{proof}

For $n$ a power of $2$, the situation is different but has been solved. To state the result, we need some notation. For $m\geq 2$, write
\begin{equation}\label{D} D_{2^m}  = \langle r,s \mid r^{2^{m-1}} = 1,\, s^2 = 1,\, srs^{-1} = r^{-1}\rangle
\end{equation}
for the dihedral group of order $2^m$, and note that $D_4$ is the Klein four-group. For $m\geq 3$, similarly write
\begin{equation}\label{Q}
Q_{2^m}  = \langle r,s\mid r^{2^{m-1}} = 1,\, s^2 =r^{2^{m-2}},\, srs^{-1} = r^{-1}\rangle
\end{equation}
for the generalized quaternion group of order $2^m$. It is known that:

\begin{prop}\label{p prop2}Let $N$ be a group of order $2^m$. Then:
\begin{enumerate}[(a)]
\item For $m\leq 2$, the pair $(C_{2^m},N)$ is always realizable.
\item For $m\geq 3$ the pair $(C_{2^m},N)$ is realizable if and only if $N\simeq  C_{2^m},D_{2^m},Q_{2^m}$.
\end{enumerate}
\end{prop}
\begin{proof}See \cite[Lemma 2]{By96} for (a) and  \cite[Corollary 5.3, Theorem 6.1]{Byott 2} for (b).
\end{proof}

By \cite[Theorem 1]{Byott nilpotent}, with Propositions \ref{p prop1} and \ref{p prop2}, we get a complete characterization of nilpotent groups $N$ for which $(C_n,N)$ is realizable. We remark that the exact number of Hopf-Galois structures of nilpotent type $N$ on any Galois $C_n$-extension is given in \cite[Theorem 2]{Byott nilpotent}. But as the next proposition shows, the pair $(C_n,N)$ can also be realizable for non-nilpotent groups $N$.

\vspace{1.5mm}

A finite group is called a \emph{$C$-group} (or \emph{$Z$-group}) if all of its Sylow subgroups are cyclic. The terminology comes from \cite{MM}, where a very nice description of $C$-groups was given. By \cite[Lemma 3.5]{MM}, every $C$-group is presented as
\[\label{C group}
C(e,d,k) =  \langle x,y \mid x^e = 1,\, y^d = 1,\, yxy^{-1} = x^{k}\rangle\]
for $\gcd(e,d) = \gcd(e,k) =1$, and the order of $k$ in $(\mathbb{Z}/e\mathbb{Z})^\times$ divides $d$. Then, it is essentially known by work in the literature that:

\begin{prop}\label{C prop}For any $C$-group $N$ of order $n$, the pair $(C_n,N)$ is realizable.
\end{prop}
\begin{proof}Since $N$ is a $C$-group, by the above $N\simeq C_e\rtimes C_d$ with $\gcd(e,d)=1$. Then, it is known and we shall also explain in Proposition \ref{product fpf} that the pair $(C_e\times C_d,N)$ is realizable. But $C_{n} \simeq C_e\times C_d$ since $n = ed$ with $\gcd(e,d)=1$, and the claim now follows.\end{proof}

For $n$ squarefree, every group $N$ of order $n$ is a $C$-group so the pair $(C_n,N)$ is always realizable. In fact, the number of Hopf-Galois structures of type $N$ on any Galois $C_n$-extension has been determined in terms of the orders of the center and commutator subgroup of $N$ (see \cite{AB}). Similarly for the number of skew braces with additive group $N$ and multiplicative group $C_n$ (see \cite{AB''}).

\vspace{1.5mm}

For $n$ arbitrary, however, not every group $N$ of order $n$ is a $C$-group and it is certainly possible that $(C_n,N)$ is realizable for a non-$C$-group $N$ because of Proposition \ref{p prop2}. The only known general restriction on $N$ so far is:

\begin{prop}\label{cyclic prop}Let $N$ be a group of order $n$ such that the pair $(C_n,N)$ is realizable. Then $N$ is both supersolvable and metabelian.
\end{prop}
\begin{proof}See \cite[Theorem 1.3(a),(b)]{Tsang solvable}.
\end{proof}

Unfortunately, the converse of Proposition \ref{cyclic prop} is false. For example, using \textsc{Magma} \cite{magma}, one can check that the group $N=\textsc{SmallGroup}(84,8)$ is both supersolvable and metabelian, yet the pair $(C_{84},N)$ is not realizable. 

\vspace{1.5mm}

In this paper, by building upon the four propositions above, we shall give a complete characterization of the groups $N$ of order $n$ for which $(C_n,N)$ is realizable, without imposing any assumption on $n$ or $N$. By Proposition \ref{C prop}, it suffices to consider non-$C$-groups $N$. Our main theorem is:

\begin{thm}\label{main thm}Let $N$ be a non-$C$-group of order $n$. Then, the pair $(C_n,N)$ is realizable if and only if $N\simeq M\rtimes_\alpha P$ for some $C$-group $M$ of odd order and $(P,\alpha)$ satisfying one of the following conditions:
\begin{enumerate}[$(1)$]
\item $P= D_4$ or $P= Q_8$, and $\alpha(P)$ has order $1$ or $2$;
\item $P= D_{2^m}$ with $m\geq 3$ or $P= Q_{2^m}$ with $m\geq 4$, and $\alpha(r)=\mathrm{Id}_M$.
\end{enumerate}
Here $\alpha : P \longrightarrow\Aut(M)$ is the homomorphism which defines the semidirect product, and $r$ is the element of $P$ in the presentation (\ref{D}) or (\ref{Q}).
\end{thm}

\begin{cor}Let $N$ be a group of order $n$ with $n$ not divisible by $4$. Then the pair $(C_n,N)$ is realizable precisely when $N$ is a $C$-group.
\end{cor}
\begin{proof}The forward implication holds by Theorem \ref{main thm} because there $M$ is a group of odd order while $P$ is a $2$-group of order at least $4$. The backward implication is Proposition \ref{C prop}. \end{proof}

\begin{remark}\label{remark Rump}Instead of fixing $G$ to be cyclic, one can also fix $N$ to be cyclic and ask for which groups $G$ is the pair $(G,C_n)$ realizable. This case has been solved completely in \cite[Corollary 1 to Theorem 2]{Rump}, which states that
\[ (G,C_n)\mbox{ is realizable}\iff G\mbox{ is solvable, $2$-nilpotent, almost Sylow-cyclic}.\]
Here $G$ being $2$-nilpotent means that it has a normal Hall $2'$-subgroup $M$. By the Schur-Zassenhaus theorem, this simply means that $G = M\rtimes P$, where $P$ denotes any Sylow $2$-subgroup of $G$. The term \emph{almost Sylow-cyclic} means that every Sylow $p$-subgroup is cyclic for odd primes $p$, and any non-trivial Sylow $2$-subgroup contains a cyclic subgroup of index $2$. We then see that the pair $(G,C_n)$ is realizable if and only if $G\simeq M\rtimes_\alpha P$, where
\begin{enumerate}[(a)]
\item $M$ is any $C$-group of odd order,
\item $P$ is trivial or any $2$-group admitting a cyclic subgroup of index $2$,
\end{enumerate}
and there is no restriction on the homomorphism $\alpha : P\longrightarrow \Aut(M)$. Note that such a group $G$ is  always solvable because $C$-groups are solvable.

\vspace{1.5mm}

Comparing this with Theorem \ref{main thm}, we deduce that realizability of $(C_n,\Gamma)$ implies that of $(\Gamma,C_n)$, but the converse fails to hold for certain values of $n$.\end{remark}

\section{Methods to study realizability}

Let $G$ and $N$ be two finite groups of the same order. Below, we review a couple of techniques that can be used to study the realizability of $(G,N)$.

\subsection{Characteristic subgroups and induction}

To prove that $(G,N)$ is not realizable, one approach is to use \emph{characteristic} subgroups of $N$, namely subgroups $M$ such that $\pi(M) = M$ for all $\pi\in\Aut(N)$. This was developed by the author in \cite[Section 4]{Tsang HG} and was inspired by work of \cite{Byott simple}.

\vspace{1.5mm}

First, recall that given $\ff\in\Hom(G,\Aut(N))$, a map $\fg: G\longrightarrow N$ is said to be a \emph{crossed homomorphism (with respect to $\ff)$} if it satisfies
\begin{equation}\label{g relation} \fg(\sigma\tau) = \fg(\sigma)\cdot \ff(\sigma)(\fg(\tau)) \mbox{ for all }\sigma,\tau\in G.\end{equation}
Let us write $Z_\ff^1(G,N)$ for the set of all such crossed homomorphisms.

\begin{prop}\label{regular lem}The regular subgroups of $\Hol(N)$ isomorphic to $G$ are exactly the subsets of $\Hol(N)$ of the form
\[\{\rho(\fg(\sigma))\cdot\ff(\sigma):\sigma\in G\},\mbox{ where }\begin{cases}
\ff\in\Hom(G,\Aut(N)),\\ \fg\in Z_\ff^1(G,N)\mbox{ is bijective}.
\end{cases}\]
\end{prop}
\begin{proof}This follows directly from the fact that $\Hol(N) = \rho(N)\rtimes \Aut(N)$; or see \cite[Proposition 2.1]{Tsang HG} for a proof.
\end{proof}

The next proposition gives us a way to show that $(G,N)$ is not realizable using characteristic subgroups of $N$ and induction.  We remark that (a) was previously known but (b) is new.

\begin{prop}\label{char prop} Let $\ff\in\Hom(G,\Aut(N))$ and let $\fg\in Z_\ff^1(G,N)$ be a bijective crossed homomorphism. Let $M$ be any characteristic subgroup of $N$ and define $H = \fg^{-1}(M)$. Then
\begin{enumerate}[$(a)$]
\item $H$ is a subgroup of $G$ and the pair $(H,M)$ is realizable;
\item $H$ is a normal subgroup of $G$ and the pair $(G/H,N/M)$ is realizable, as long as $H$ lies in the center $Z(G)$ of $G$.\end{enumerate}
\end{prop}
\begin{proof}By (\ref{g relation}) and the fact that $M$ is a characteristic subgroup of $N$, plainly $H$ is a subgroup of $G$, which has the same order as $M$ since $\fg$ is bijective. 

\vspace{1.5mm}

That $(H,M)$ is realizable was shown in \cite[Proposition 3.3]{Tsang solvable}. The idea was that via restriction $\ff$ induces a homomorphism
\[ \underline{\ff}_M \in \Hom(H,\Aut(M)); \hspace{1em}\underline{\ff}_M(\tau) = (\eta\mapsto\ff(\tau)(\eta))\]
since $M$ is characteristic, and $\fg$ induces a bijective crossed homomorphism
\[ \underline{\fg}_M\in Z^1_{\underline{\ff}_M}(H,M);\hspace{1em}\underline{\fg}_M(\tau) = \fg(\tau)\]
since $M =\fg(H)$. From Proposition \ref{regular lem}, we then obtain a regular subgroup
of $\Hol(M)$ isomorphic to $H$, whence $(H,M)$ is realizable.

\vspace{1.5mm}

Suppose now that $H$ lies in $Z(G)$. It is clear that $H$ is a normal subgroup of $G$. First, we show that $\ff$ induces a well-defined homomorphism
\[ \overline{\ff}_M\in \Hom(G/H,\Aut(N/M));\hspace{1em}\overline{\ff}_M(\sigma H) = (\eta M\mapsto \ff(\sigma)(\eta)M).\]
For any $\sigma\in G$ and $\tau\in H$, since $H$ lies in $Z(G)$, by (\ref{g relation}) we have
\[ \fg(\tau)\cdot \ff(\tau)(\fg(\sigma)) = \fg(\tau\sigma) = \fg(\sigma\tau) = \fg(\sigma)\cdot \ff(\sigma)(\fg(\tau)).\]
But $M = \fg(H)$ is characteristic, so reducing mod $M$ then yields
\[ \ff(\tau)(\fg(\sigma)) \equiv \fg(\sigma)\hspace{-2mm}\pmod{M}.\]
Since $\fg$ is bijective, this implies that $\ff(\tau)$ induces the identity automorphism on $N/M$. This holds for all $\tau\in H$, whence $\overline{\ff}_M$ is indeed well-defined. Similarly $\fg$ induces a bijective crossed homomorphism
\[ \overline{\fg}_M\in Z^1_{\overline{\ff}_M}(G/H,N/M);\hspace{1em}\overline{\fg}_M(\sigma H) = \fg(\sigma)M,\]
which is also well-defined by (\ref{g relation}) because $M = \fg(H)$ is characteristic. From Proposition \ref{regular lem}, we then get a regular subgroup
of $\Hol(N/M)$ isomorphic to $G/H$, whence $(G/H,N/M)$ is realizable.
\end{proof}

\subsection{Fixed point free pairs of homomorphisms}\label{fpf sec} To prove that $(G,N)$ is realizable, one approach is to use homomorphisms $f,h\in\Hom(G,N)$ such that $(f,h)$ is \emph{fixed point free}, namely $f(\sigma) = h(\sigma)$ if and only if $\sigma = 1_G$. This was introduced by N. P. Byott and  L. N. Childs in \cite{BC}.

\begin{prop}\label{fpf prop}Let there exist $f,h\in \Hom(G,N)$ such that $(f,h)$ is fixed point free. Then $(G,N)$ is realizable.
\end{prop}
\begin{proof}Since elements in $\lambda(N)$ and $\rho(N)$ commute, plainly
\[ \{ \rho(h(\sigma)) \lambda(f(\sigma)) : \sigma \in G\}\]
is a subgroup of $\Hol(N)$ isomorphic to $G$, whose regularity follows from the fixed-point freeness of $(f,h)$; see \cite[Proposition 1]{BC} for a proof. We note that in the notation of Proposition \ref{regular lem}, this corresponds to
\[ \ff \in \Hom(G,\Aut(N));\hspace{1em} \ff(\sigma) = \conj(f(\sigma)),\]
where $\conj(\eta) = (x\mapsto \eta x \eta^{-1})$ denotes conjugation by $\eta$, and
\[ \fg \in Z_{\ff}^1(G,N);\hspace{1em} \fg(\sigma) = h(\sigma)f(\sigma)^{-1},\]
which is bijective because $(f,h)$ is fixed point free.
\end{proof}

The next proposition is from \cite[Lemma 7.1]{Byott soluble}. 


\begin{prop}\label{product fpf}Suppose that $N = N_1N_2$ for subgroups $N_1$ and $N_2$ such that $N_1\cap N_2=1$. Then $(N_1\times N_2,N)$ is realizable.
\end{prop}
\begin{proof}This follows from Proposition \ref{fpf prop} and the obvious fact that $(f,h)$ is a fixed point free pair for $f,h\in\Hom(N_1\times N_2,N)$ defined by $f(\eta_1,\eta_2) = \eta_1$ and $h(\eta_1,\eta_2) = \eta_2$.\end{proof}

As noted in Proposition \ref{C prop}, an easy application of Proposition \ref{product fpf} shows that $(C_n,N)$ is always realizable for $C$-groups $N$ of order $n$. However, as the next proposition shows, there is no fixed point free pair of homomorphisms from $C_n$ to $N$ for non-$C$-groups $N$. Therefore, we cannot simply use Proposition \ref{fpf prop} to prove realizability in Theorem \ref{main thm}. Instead, we shall exhibit the existence of a cyclic regular subgroup in $\Hol(N)$ using a direct approach.

\begin{prop}\label{no fpf prop}Let $N$ be a group of order $n$ such that there exists a fixed point free pair $(f,h)$ with $f,h\in\Hom(C_n,N)$. Then $N$ is a $C$-group.
\end{prop}
\begin{proof}Let $\sigma$ be a generator of $C_n$, and put
\[ d_f = |f(\sigma)|,\,\ d_h = |h(\sigma)|,\,\ g = \gcd(d_f,d_h).\]
Then $\sigma^{d_fd_h/g} =1_G$ because $(f,h)$ is fixed point free and
\[ f(\sigma)^{d_f(d_h/g)} = 1_N = h(\sigma)^{d_h(d_f/g)}.\]
But $d_fd_h/g$ divides $n$ because both $d_f,d_h/g$ divide $n$ and $\gcd(d_f,d_h/g) =1$. It follows that $d_fd_h / g =n$ and so $n = \mathrm{lcm}(d_f,d_h)$.
Hence, we may write
\begin{align*}
 d_f &= p_1^{e_1}\cdots p_a^{e_a} g_f,\\
  d_h &= p_{a+1}^{e_{a+1}}\cdots p_{b}^{e_b}g_h,\end{align*}
where $p_1,\dots,p_a,p_{a+1},\dots,p_b$ are distinct primes and $g_f,g_h\in\mathbb{N}$, such that
\[ n = p_1^{e_1}\cdots p_a^{e_a}p_{a+1}^{e_{a+1}}\cdots p_b^{e_b}\]
is the prime factorization of $n$. Then
\[ |f(\sigma)^{g_f}| = p_1^{e_1}\cdots p_a^{e_a},\,\ |h(\sigma)^{g_h}| = p_{a+1}^{e_{a+1}}\cdots p_b^{e_b},\,\ |f(\sigma)^{g_f}||h(\sigma)^{g_h}| =n.\]
We deduce that $N =  \langle f(\sigma)^{g_f} \rangle\langle h(\sigma)^{g_h}\rangle$ is the product of two cyclic subgroups of coprime orders, and thus $N$ is a $C$-group.
\end{proof}

\section{Preliminary restriction}

Let us first prove a preliminary version of Theorem \ref{main thm}:

\begin{thm}\label{pre thm}Let $N$ be a group of order $n$ such that $(C_n,N)$ is realizable. Then either $N$ is a $C$-group or $N\simeq M\rtimes P$ for some $C$-group $M$ of odd order and $P= D_{2^m}$ with $m\geq 2$ or $P= Q_{2^m}$ with $m\geq 3$.
\end{thm}


\begin{proof}Let $n = p_1^{e_1}\cdots p_b^{e_b}$ be the prime factorization of $n$ with $p_1>\cdots >p_b$. For each $1\leq a \leq b$, let $P_a$ be a Sylow $p_a$-subgroup of $N$. Put
\[ M = P_1\cdots P_{b-1}\mbox{ and }P = P_b.\]
Recall that $N$ has to be supersolvable by Proposition \ref{cyclic prop}. Then, it is known (see \cite[Corollary VII.5.h]{GT book} for example) that $M$ is a normal subgroup of $N$, and we have $N = M\rtimes P$. 
But plainly $M$ is a characteristic subgroup of $N$, so by Proposition \ref{char prop}, there is a subgroup $H$ of $C_n$ (of the same order as $M$) such that the pairs $(H,M)$ and $(C_n/H,N/M)$ are both realizable. Note that 
\[ H \simeq C_{p_1^{e_1}\cdots p_{b-1}^{e_{b-1}}} \mbox{ and }C_n/H \simeq C_{p_b^{e_b}}\]
are both cyclic. Thus, we may prove the claim using induction on $b$.

\vspace{1.5mm}

First, consider the case when $n$ is odd. For $b=1$, we know by Proposition \ref{p prop1} that $N\simeq C_{p_1^{e_1}}$ and hence is a $C$-group. For $b\geq2$, by induction we may assume that $M$ is a $C$-group, which implies that $P_1,\dots,P_{b-1}$ are all cyclic. But $P \simeq N/M$ is also cyclic by Proposition \ref{p prop1}, whence $N$ is a $C$-group.

\vspace{1.5mm}

Next, consider the case when $n$ is even, so then $p_b=2$. Since $M$ has odd order, we already know that $M$ must be a $C$-group. If $P$ is cyclic, then $N$ is a $C$-group as above. If $P\simeq N/M$ is non-cyclic, then necessarily
\[ P\simeq \begin{cases}D_4&\mbox{when }e_b=2,\\  D_{2^{e_b}},Q_{2^{e_b}}&\mbox{when }e_b\geq 3\end{cases}\]
by Proposition \ref{p prop2}. This completes the proof of the theorem.
\end{proof}

\begin{remark}\label{pre remark}
The converse of Theorem \ref{pre thm} is not true. For example, as mentioned in the introduction, the pair $(C_{84},N)$ is not realizable for 
\[ N = \textsc{SmallGroup}(84,8)\mbox{ but }N\simeq C_{21}\rtimes_\alpha D_4,\]
as one can check in \textsc{Magma} \cite{magma}. Alternatively, this group $N$ corresponds to when $\alpha : D_4\longrightarrow \Aut(C_{21})$ embeds $D_4$ into the unique Sylow $2$-subgroup of $\Aut(C_{21})$. One sees that $N\simeq D_{14}\times D_{6}$ in this case. Since both factors $D_{14}$ and $D_{6}$ are characteristics, we have
\[ \Hol(N) \simeq \Hol(D_{14})\times \Hol(D_6).\]
The structure of the automorphism groups of dihedral groups is well-known (see \cite[Theorem 1.4]{Conrad D} for example). It is not hard to see that $\Hol(D_{14})$ and $\Hol(D_{6})$ do not have any elements of order $4$. This implies that $\Hol(N)$ does not even have a cyclic subgroup of order $84$, let alone a regular one. Hence, indeed $(C_{84},N)$ is not realizable.

\vspace{1.5mm}

\end{remark}

\section{Groups of the shape $M\rtimes_\alpha P$}

Throughout this section, let $M$ denote the $C$-group
\[ C(e,d,k) = \langle x,y \mid x^e = 1,\, y^d = 1,\, yxy^{-1} = x^{k}\rangle\]
for $\gcd(e,d) = \gcd(e,k) = 1$, and the order $\mbox{ord}_e(k)$ of $k$ in $(\mathbb{Z}/e\mathbb{Z})^\times$ divides $d$. Also, let $P$ denote the dihedral group
\[D_{2^m} = \langle r,s \mid r^{2^{m-1}} =1,\, s^2 = 1,\, srs^{-1} = r^{-1}\rangle\]
with $m\geq 2$ or the generalized quaternion group
\[ Q_{2^m} = \langle r,s \mid  r^{2^{m-1}} = 1, \, s^2 = r^{2^{m-2}},\, srs^{-1}=r^{-1}\rangle\]
with $m\geq 3$. In order to prove Theorem \ref{main thm}, we shall need to understand the structure of the semidirect products $M\rtimes_\alpha P$ for $\alpha\in \Hom(P,\Aut(M))$.

\subsection{Automorphism group of $C$-groups}\label{Aut sec}

Let us first determine the automorphism group $\Aut(M)$ of $M$ in a way that is analogous to \cite[Lemma 4.1]{AB}, which treats the special case when $ed$ is squarefree. 

\vspace{1.5mm}
For $h\in \mathbb{Z}$ and $\ell\in \mathbb{N}_{\geq 0}$, let us define
\[ S(h,\ell) = \sum_{a= 0}^{\ell-1} h^a = 1 + h + \cdots + h^{\ell-1},\]
with the empty sum $S(h,0)$ representing zero. For $i,j\in\mathbb{Z}$, a simple calculation using induction on $\ell$ and the relation $yxy^{-1} = x^k$ yields
\[ (x^iy^j)^\ell = x^{iS(k^j,\ell)}y^{j\ell}.\]
We shall use this identity without reference in what follows. Also put
\[ z = \gcd(e,k-1)\mbox{ and }g = e/z.\]
Further, consider the multiplicative groups
\begin{align*}
 U(e) = (\mathbb{Z}/e\mathbb{Z})^\times\mbox{ and }
 U_k(d) = \{v \in (\mathbb{Z}/d\mathbb{Z})^\times \mid v\equiv1\hspace{-4mm}\pmod{\mathrm{ord}_e(k) }\}.
 \end{align*}
Recall that $\mbox{ord}_e(k)$ denotes the order of $k$ in $(\mathbb{Z}/e\mathbb{Z})^\times$ and it divides $d$.

\begin{lem}\label{Aut def} For any $u\in U(e)$ and $v\in U_k(d)$, the definitions
\[\begin{cases}\theta(x) = x \\ \theta(y) = x^zy
\end{cases}
\begin{cases}\phi_u(x) = x^u\\ \phi_u(y) = y
\end{cases}
\begin{cases}\psi_v(x) = x\\ \psi_v(y) = y^v
\end{cases}\]
extend to automorphisms on $M$. Moreover, we have the relations
\begin{equation}\label{relations} 
\theta^g = \mathrm{Id}_M,\,\ \phi_u\theta\phi_u^{-1} = \theta^u,\,\ \theta\psi_v = \psi_v\theta,\,\ \phi_u\psi_v = \psi_v\phi_u.\end{equation}
\end{lem}
\begin{proof}We may assume that $k\neq1$, for otherwise $M\simeq C_e \times C_d$ (with $z=e$ and so $\theta$ is the identity), in which case all of the claims are trivial. 

\vspace{1.5mm}

First, we check that the three relations 
\[x^e = 1,\, y^d = 1,\, yxy^{-1} = x^k\]
in the presentation of $M$ are preserved under these maps. Clearly
\[ \theta(x)^e = \phi_u(x)^e = \psi_v(x)^e = 1\mbox{ and }\phi_u(y)^d = \psi_v(y)^d = 1\]
are satisfied. We compute that
\[ \theta(y)^d = (x^zy)^d = x^{zS(k,d)}y^d = x^{zS(k,d)}.\]
Since $\mathrm{ord}_e(k)$ divides $d$, we have
\[ \left(\frac{k-1}{z}\right)zS(k,d) \equiv k^d -1 \equiv 0\pmmod{e}.\]
But $\gcd(\frac{k-1}{z},e)=1$, so then $zS(k,d)\equiv 0\pmod{e}$ and we obtain $\theta(y)^d=1$. A simple calculation also yields
\begin{align*}
\theta(y)\theta(x)\theta(y)^{-1} & = (x^zy) x (x^zy)^{-1} = x^k = \theta(x)^k,\\
\phi_u(y)\phi_u(x)\phi_u(y)^{-1} & = y x^u y^{-1} = x^{uk} = \phi_u(x)^k,\\
\psi_v(y)\psi_v(x)\psi_v(y)^{-1} & = y^v x y^{-v} = x^{k^v}  = x^k = \psi_v(x)^k,
\end{align*}
where $x^{k^v}=x^k$ holds because $v\in U_k(d)$ implies $k^v\equiv k\pmod{e}$. Thus, all of $\theta,\phi_u,\psi_v$ extend to endomorphisms on $M$. It is clear that their images all contain $x,y$, so in fact $\theta,\phi_u,\psi_v$ extend to automorphisms on $M$.

\vspace{1.5mm}

Next, we verify the relations in (\ref{relations}). The first and last equalities are both obvious. For the second equality, a simple calculation shows that
\[ (\phi_u\theta)(x) = x^u = (\theta^u\phi_u)(x)\mbox{ and } (\phi_u\theta)(y) = x^{uz}y = (\theta^u\phi_u)(y).\]
For the third equality, plainly $(\theta \psi_v)(x) = x = (\psi_v\theta)(x)$. We also have
\[ (\theta\psi_v)(y) = (x^zy)^v = x^{zS(k,v)}y^v \mbox{ and }
 (\psi_v\theta)(y) = x^{z}y^v.\]
But $v\in U_k(d)$ implies that $k^v \equiv k\pmod{e}$, so then
\[\left(\frac{k-1}{z}\right)zS(k,v)\equiv k^v - 1\equiv k-1\equiv \left(\frac{k-1}{z}\right)z\pmmod{e}.\]
Since $\gcd(\frac{k-1}{z},e)=1$, this implies that
\begin{equation}\label{z mod e} zS(k,v)\equiv z\pmmod{e}\mbox{ and hence }(\theta\psi_v)(y) = (\psi_v\theta)(y).\end{equation}
It follows that $\theta\psi_v=\psi_v\theta$, as desired.
\end{proof}

\begin{prop}\label{Aut prop}We have 
\[ \Aut(M) = \left(\langle\theta\rangle \rtimes \{\phi_u\}_{u\in U(e)}\right)\times \{\psi_v\}_{v\in U_k(d)}.\]
\end{prop}
\begin{proof} It is easy to check that the three subgroups
\[\langle\theta\rangle,\,\ \{\phi_u\}_{u\in U(e)},\,\ \{\psi_v\}_{v\in U_k(d)}\]
have trivial pairwise intersections. By the relations in (\ref{relations}),  it is then enough to show that every $\pi \in \Aut(M)$ lies in their product.

\vspace{1.5mm}

First, since $\gcd(e,d)=1$, clearly
\[ \pi(x) = x^u \mbox{ with }u\in U(e),\mbox{ and let us write }\pi(y) = x^cy^v.\]
We must have $\gcd(v,d)=1$, for otherwise there would exist $\ell\in\mathbb{N}$ which is strictly less than $d$ such that $d$ divides $v\ell$, and
\[ \pi(y)^\ell = (x^cy^v)^\ell = x^{cS(k^v,\ell)}y^{v \ell } = x^{cS(k^v,\ell)}.\]
But $\langle \pi(y)\rangle$, which has order $d$, cannot contain a non-trivial element of order dividing $e$ because $\gcd(e,d)=1$. This then implies that $\pi(y)^\ell = 1$, which is impossible since $1\leq \ell \leq d-1$. Next, observe that
\[ x^{uk^v} = (x^cy^v)x^u(x^cy^v)^{-1}= \pi(y)\pi(x)\pi(y)^{-1} = \pi(x)^k = x^{uk}.\]
Since $\gcd(u,e)=1$, it follows that
\[ k^v\equiv k\pmmod{e},\mbox{ and hence } v\in U_k(d).\]
We also have the equalities
\[ 1 = \pi(y)^d = (x^cy^v)^d = x^{cS(k^v,d)}y^{dv} = x^{cS(k^v,d)}.\]
Recall that $z=\gcd(e,k-1)$. Then, the above in particular implies that
\[ cd\equiv cS(1,d)\equiv cS(k^v,d)\equiv0\pmmod{z},\]
and so $c$ is divisible by $z$ because $\gcd(z,d)=1$.

\vspace{1.5mm}

Finally, we compute that
\begin{align*}
(\theta^{\frac{c}{z}}\phi_u\psi_v)(x) &= (\theta^{\frac{c}{z}}\phi_u)(x) = \theta^{\frac{c}{z}}(x^u) = x^u, \\
(\theta^{\frac{c}{z}}\phi_u\psi_v)(y) &= (\theta^{\frac{c}{z}}\phi_u)(y^v) = \theta^{\frac{c}{z}}(y^v) = (x^{\frac{c}{z}\cdot z}y)^v = x^{\frac{c}{z}\cdot zS(k,v)}y^v = x^{c}y^v,
\end{align*}
where the last equality holds by the congruence in (\ref{z mod e}). It thus follows that $\pi = \theta^{\frac{c}{z}}\phi_u \psi_v$, and this completes the proof. 
\end{proof}

\subsection{Dihedral and generalized quaternion groups} Let us record a few  \\\par\noindent facts that we shall need concerning the commutator subgroup $P'$ of $P$, and also the automorphism group $\mathrm{Aut}(P)$ of $P$.

\begin{lem}\label{commutator lem}We have $P'=\langle r^2\rangle$ and $P/P'\simeq D_4$.
\end{lem}
\begin{proof}We have $r^2\in P'$ because $srs^{-1}r^{-1} = r^{-2}$, and clearly $\langle r^2\rangle$ is a normal subgroup of order $2^{m-2}$. Since $P/\langle r^2\rangle$ has order $4$, whose exponent is easily seen to be $2$, we must have $P/\langle r^2\rangle \simeq D_4$. The fact that $r^2\in P'$ and $P/\langle r^2\rangle$ is abelian implies that $P' =\langle r^2\rangle$.
\end{proof}


\begin{prop}\label{DQ prop}The following hold.
\begin{enumerate}[$(a)$]
\item The definitions
\[\begin{cases}\kappa_1(r) =s\\ \kappa_1(s) =r
\end{cases}\mbox{ and }\hspace{1em}\begin{cases}\kappa_2(r) =rs\\ \kappa_2(s) =s
\end{cases}\]
extend to automorphisms on $D_4$. 
\item The definitions
\[\begin{cases}\kappa_1(r) =s\\ \kappa_1(s) =rs^2
\end{cases}\mbox{ and }\hspace{1em}\begin{cases}\kappa_2(r) =rs\\ \kappa_2(s) = r
\end{cases}\]
extend to automorphisms on $Q_8$.
\item Assume that $P = D_{2^m}$ with $m\geq 3$ or $P= Q_{2^m}$ with $m\geq 4$. Then, for any $a,b\in\mathbb{Z}$ with $a$ odd, the definition
\[\begin{cases}
\kappa(r) = r^a\\
\kappa(s) = r^bs
\end{cases}\]
extends to an automorphism on $P$. Conversely, all automorphisms on $P$ arise in this way.
\end{enumerate}
\end{prop}
\begin{proof}Part (a) is obvious and part (b) follows from a simple calculation. As for part (c), see \cite[Theorem 1.4]{Conrad D} and \cite[Theorem 4.7]{Conrad Q}.
\end{proof}

\begin{remark}In Proposition \ref{DQ prop}, the $\kappa_1,\kappa_2$ in (a) do not extend to automorphisms on $D_{2^m}$ for $m\geq 3$, and those in (b) do not extend to automorphisms on $Q_{2^m}$ for $m\geq 4$. This difference is the reason why there are two cases to consider in Theorem \ref{main thm}.
\end{remark}

\subsection{Properties of the homomorphism $\alpha$}\label{N prop sec}

Let $\alpha\in\mathrm{Hom}(P,\Aut(M))$ be fixed, and let $N = M\rtimes_\alpha P$ be the semidirect product defined by $\alpha$. For each $t\in P$, let us write $\alpha_t = \alpha(t)$ for short. Then, in the group $N$ we have
\[ t x t^{-1} = \alpha_t(x)\mbox{ and } tyt^{-1} = \alpha_t(y).\]
We shall study properties of $\alpha$ using results from the previous subsections.


\begin{assumption} We shall assume that the order $ed$ of $M$ is odd since this is the only case of interest for us. In the presentation of $M$, by \cite{MM} without loss of generality, we may assume that $\mbox{ord}_e(k)$, which has to divide $d$, is divisible by all prime factors of $d$. 
\end{assumption}

\begin{lem}\label{alpha lemma} The homomorphism $\alpha$ satisfies the following:
\begin{enumerate}[$(a)$]
\item $\alpha(P)$ lies in $\langle\theta\rangle\rtimes \{\phi_u\}_{u\in U(e)}$;
\item $\ker(\alpha)$ contains $P'$;
\item $\alpha(P)$ is elementary $2$-abelian of order $1,2$, or $4$;
\item $\alpha_{t_1}(x) = \alpha_{t_2}(x)$ implies $\alpha_{t_1}=\alpha_{t_2}$ for any $t_1,t_2\in P$.
\end{enumerate}
\end{lem}
\begin{proof}Since $\mathrm{ord}_e(k)$ is divisible by all prime factors of $d$, the order of $U_k(d)$ divides $d$ and hence is odd. Since $P$ is a $2$-group, the projection of $\alpha(P)$ onto $\{\psi_v\}_{v\in U_k(d)}\simeq U_k(d)$ must then be trivial. This gives (a).

\vspace{1.5mm}

The order of $\langle\theta\rangle$ divides $e$ by (\ref{relations}) and thus is also odd. This means that $\{\phi_u\}_{u\in U(e)}$ contains a Sylow $2$-subgroup of $\Aut(M)$. Since $\{\phi_u\}_{u\in U(e)}\simeq U(e)$ is abelian, this implies that $\alpha(P)$ must be abelian. This proves (b), and (c) follows as well by Lemma \ref{commutator lem}.

\vspace{1.5mm}

Let $t_1,t_2\in P$ be such that $\alpha_{t_1}(x) = \alpha_{t_2}(x)$. By (a), we may write
\[ \alpha_{t_1} = \theta^{c_1}\phi_{u_1}\mbox{ and }\alpha_{t_2}=\theta^{c_2}\phi_{u_2},\mbox{ where }c_1,c_2\in\mathbb{Z},\, u_1,u_2\in U(e).\]
That $\alpha_{t_1}(x) = \alpha_{t_2}(x)$ means $x^{u_1} = x^{u_2}$ and hence $\phi_{u_1} = \phi_{u_2}$. By (c), we know that $\alpha_{t_1},\alpha_{t_2}$ have order dividing $2$ and they commute. It follows that
\[ \alpha_{t_1}\cdot \alpha_{t_2}^{-1} = \theta^{c_1}\phi_{u_1}\cdot \phi_{u_2}^{-1}\theta^{-c_2} = \theta^{c_1-c_2}\]
also has order dividing $2$. But $\theta$ has odd order, so we have $\theta^{c_1} = \theta^{c_2}$. Thus, indeed $\alpha_{t_1} = \alpha_{t_2}$, and this proves (d).
\end{proof}

Before proceeding, let us make two observations. First, recall that $P' = \langle r^2\rangle$ by Lemma \ref{commutator lem}, and that $\ker(\alpha)$ contains $P'$ by Lemma \ref{alpha lemma}(b). It then follows that $\ker(\alpha)$ is equal to one of the following:
\begin{equation}\label{ker poss}
 \langle r^2\rangle ,\,\ \langle r^2,s\rangle,\,\ \langle r^2,rs\rangle,\,\ \langle r\rangle,\,\ \langle r,s\rangle.
 \end{equation}
For these five possibilities, the order of $\alpha(P)$ is respectively given by
\[ 4,\, 2,\, 2,\, 2,\, 1.\]
Second, notice that $M$, whose order is assumed to be odd, is a characteristic subgroup of $N$. Then $\langle x\rangle$, being characteristic in $M$ because $\gcd(e,d)=1$, is also a characteristic and in particular normal subgroup of $N$.

\begin{lem}\label{order 2 lem}Elements in $N$ of order a power of $2$ all lie in $\langle x\rangle \rtimes_\alpha P$.\end{lem}
\begin{proof}
Let $x^iy^jt\in N$ be of order $2^\ell$ with $t\in P$. By Lemma \ref{alpha lemma}(a), we have
 \[ \alpha_t(y) \equiv y\pmmod{\langle x\rangle},\]
so then $y$ and $t$ commute modulo $\langle x\rangle$. It follows that
\[ y^{2^\ell j} t^{2^\ell}  \equiv (y^jt)^{2^\ell} \equiv (x^iy^jt)^{2^\ell} \equiv 1\pmmod{\langle x\rangle}.\]
But then $y^{2^\ell j} =1$, which implies that $y^j = 1$ because $y$ has odd order. Thus, indeed $x^iy^jt = x^it$ belongs to $\langle x\rangle\rtimes_\alpha P$.
\end{proof}

To prove necessity in Theorem \ref{main thm}, consider the natural homomorphism
\begin{equation}\label{mod M}
\begin{tikzcd}[column sep = 3.5cm]
\Aut(N)\arrow{r}{\xi\mapsto (\eta M\mapsto \xi(\eta)M)} &\Aut(N/M)\arrow[equal]{r}{\text{identification}} & \Aut(P).
\end{tikzcd}
\end{equation}
We shall require the next proposition.

\begin{prop}\label{P' prop}Let $\kappa\in\Aut(P)$ be in the image of (\ref{mod M}).
\begin{enumerate}[$(a)$]
\item We always have 
\[ \begin{cases}
\kappa(r)\equiv r\pmmod{\ker(\alpha)},\\
\kappa(s) \equiv s\pmmod{\ker(\alpha)}.
\end{cases}\]
\item Assume that $P=D_{2^m}$ with $m\geq 3$ or $P=Q_{2^m}$ with $m\geq 4$. Then
\[ \begin{cases}
\kappa(r)\equiv r\pmmod{P'},\\
\kappa(s) \equiv s\pmmod{P'},
\end{cases}\]
provided that $\alpha_r\neq\mathrm{Id}_M$.
\end{enumerate}
\end{prop}
\begin{proof}[Proof of (a)] By Lemma \ref{alpha lemma}(d), it suffices to show that 
\begin{equation}\label{alpha eqn}
\alpha_{\kappa(r)}(x) =  \alpha_r(x)\mbox{ and }\alpha_{\kappa(s)}(x) = \alpha_s(x).\end{equation}
Let $\xi \in \Aut(N)$ be such that its image under (\ref{mod M}) equals $\kappa$. Since $\xi(P)$ lies in $\langle x\rangle\rtimes_\alpha P$ by Lemma \ref{order 2 lem}, we may write
\[ \xi(r) = x^{i_1}\kappa(r)\mbox{ and }\xi(s) = x^{i_2}\kappa(s).\]
Since $\langle x\rangle$ is characteristic in both $M$ and $N$, we also have
\[\alpha_t(x)\in\langle x\rangle\mbox{ for all $t\in P$ and }\xi(x) = x^u\mbox{ for some }u\in U(e).\]
Now, applying $\xi$ to the relation $r x r^{-1} = \alpha_r(x)$ yields
\[x^{i_1}\kappa(r)\cdot x^u \cdot \kappa(r)^{-1}x^{-i_1} = \alpha_r(x)^u\mbox{ and so }\alpha_{\kappa(r)}(x^u) = \alpha_r(x^u).\]
Similarly, applying $\xi$ to the relation $s x s^{-1} = \alpha_s(x)$ yields
\[x^{i_2}\kappa(s)\cdot x^u \cdot \kappa(s)^{-1}x^{-i_2} = \alpha_s(x)^u\mbox{ and so }\alpha_{\kappa(s)}(x^u) = \alpha_s(x^u).\]
Since $\gcd(u,e)=1$, it follows that (\ref{alpha eqn}) indeed holds, as desired.
\end{proof}
\begin{proof}[Proof of (b)] Since $P=D_{2^m}$ with $m\geq 3$ or $P=Q_{2^m}$ with $m\geq 4$, we know from Proposition \ref{DQ prop}(c) that\[\kappa(r) = r^a\mbox{ and }\kappa(s) = r^bs\]
with $a$ odd. We then have $\kappa(r)r^{-1} \in P'$ because $a-1$ is even. We also have $\kappa(s)s^{-1} \in \ker(\alpha)$ by (a).
In the case that $\alpha_r\neq\mathrm{Id}_M$, the last two possibilities in (\ref{ker poss}) are ruled out. Thus, for $\kappa(s)s^{-1}$ to lie in $\ker(\alpha)$, necessarily $b$ is even, which means that $\kappa(s)s^{-1}\in P'$ as well. This completes the proof.
\end{proof}

To prove sufficiency in Theorem \ref{main thm}, we first show that $\alpha$ may be modified to satisfy certain nice conditions. 


\begin{prop}\label{iso prop}The following hold.
\begin{enumerate}[(a)]
\item  
Assume that $P=D_4$ or $P=Q_8$, and $\alpha(P)$ has order $1$ or $2$. Then there exists $\beta\in\Hom(P,\Aut(M))$ with $\beta_r =\mathrm{Id}_M$ such that $N\simeq M\rtimes_\beta P$.
\item There always exists $\beta \in\Hom(P,\Aut(M))$ with $\beta_s\in \{\phi_u\}_{u\in U(e)}$ such that $\alpha_t,\beta_t$ are conjugates in $\Aut(M)$ for all $t\in P$ and $N\simeq M\rtimes_{\beta} P$.
\end{enumerate}
\end{prop}

\begin{proof}[Proof of (a)] Since $\alpha(P)$ has order $1$ or $2$, from (\ref{ker poss}) we see that
\[\alpha_\epsilon=\mathrm{Id}_M\mbox{ for at least one }\epsilon \in \{ r,s,rs\}.\]
Since $P=D_4$ or $P=Q_8$, by Proposition \ref{DQ prop}(a),(b), there exists $\kappa\in\Aut(P)$ such that $\kappa(r) = \epsilon$. Let us take
\[ \beta \in \Hom(P,\Aut(M));\hspace{1em}\beta(t) = \alpha(\kappa(t)).\]
Then clearly $\beta_{r} = \alpha_\epsilon =  \mathrm{Id}_M$. To show that $N\simeq M\rtimes_\beta P$, define
\[\begin{cases}
\xi(\eta) = \eta & \mbox{for }\eta\in M,\\
\xi(t) = \kappa^{-1}(t)&\mbox{for }t\in P,
\end{cases}\]
where the inputs
are regarded as elements of $N$ and the outputs as elements of $M\rtimes_\beta P$. The relation $t\eta t^{-1}=\alpha_t(\eta)$ in $N$ is preserved under $\xi$ because
\[ \xi(t)\xi(\eta)\xi(t)^{-1} = \kappa^{-1}(t) \eta \kappa^{-1}(t)^{-1} = \beta_{\kappa^{-1}(t)}(\eta)  = \alpha_t(\eta) = \xi(\alpha_t(\eta)).\]
It follows that $\xi$ extends to a homomorphism from $N$ to $M\rtimes_\beta P$, which is easily seen to be an isomorphism.
\end{proof}

\begin{proof}[Proof of (b)] We saw in the proof of Lemma \ref{alpha lemma}(b) that $\{\phi_u\}_{u\in U(e)}$ contains a Sylow $2$-subgroup of $\Aut(M)$. Since the order of $\alpha_s$ divides $4$, there exists $\pi\in \Aut(M)$ such that $\pi \alpha_s\pi^{-1}\in\{\phi_u\}_{u\in U(e)}$. Let us take
\[ \beta \in \Hom(P,\Aut(M));\hspace{1em}\beta(t) = \pi \alpha(t) \pi^{-1}.\]
Then clearly $\beta_s = \pi \alpha_s\pi^{-1} \in \{\phi_u\}_{u\in U(e)}$. To show that $N\simeq M\rtimes_\beta P$, define
\[\begin{cases}
\xi(\eta) = \pi(\eta) & \mbox{for }\eta\in M,\\
\xi(t) = t&\mbox{for }t\in P,
\end{cases}\]
where the inputs are regarded as elements of $N$ and the outputs as elements of $M\rtimes_\beta P$. The relation $t\eta t^{-1}=\alpha_t(\eta)$ in $N$ is preserved under $\xi$ because
\[ \xi(t)\xi(\eta)\xi(t)^{-1} = t \pi(\eta) t^{-1} = \beta_t(\pi(\eta)) = \pi(\alpha_t(\eta)) = \xi(\alpha_t(\eta)).\]
It follows that $\xi$ extends to a homomorphism from $N$ to $M\rtimes_\beta P$, which is easily seen to be an isomorphism.
\end{proof}

\begin{prop}\label{exists prop}Assume that $\alpha_r=\mathrm{Id}_M$ and $\alpha_s\in\{\phi_u\}_{u\in U(e)}$. Then 
\[ \xi(\eta) = (\alpha_s\phi_{k}^{-1})(\eta)\mbox{ for }\eta\in M,\,\ \xi(r) = r^{-1},\,\ \xi(s) = rs\]
extend to an automorphism on $N$ of order dividing $2d$, and 
\begin{equation}\label{N set} N = \{ \eta_0\xi(\eta_0)\cdots \xi^{\ell-1}(\eta_0):\ell\in\mathbb{N}\}\mbox{ with }\eta_0\xi(\eta_0)\cdots \xi^{n-1}(\eta_0)=1\end{equation}
for the element $\eta_0 = xyrs$ and for $n = 2^med$.
\end{prop}
\begin{proof}First, a straightforward calculation shows that the relations in $P$ are preserved under $\xi$ (c.f. Proposition \ref{DQ prop}(c)). Put $\pi = \alpha_s\phi_k^{-1}$. That $\alpha_r = \mathrm{Id}_M$ implies the relation $r \eta r^{-1} = \alpha_r(\eta) = \eta$ is preserved under $\xi$ because
\[ \xi(r) \xi(\eta) \xi(r)^{-1} = r^{-1} \pi(\eta) r = \alpha_{r^{-1}}(\pi(\eta)) = \pi(\eta) = \xi(\eta).\]
Similarly, that $\alpha_s\in \{\phi_u\}_{u\in U(e)}$ implies $\alpha_s$ and $\pi$ commute, so $s\eta s^{-1} = \alpha_s(\eta)$ is also preserved under $\xi$ because then
\[ \xi(s) \xi(\eta) \xi(s)^{-1} = rs \pi(\eta) s^{-1} r^{-1} = (\alpha_r\alpha_s\pi)(\eta) = (\pi\alpha_s)(\eta) = \xi(\alpha_s(\eta)).\]
It follows that $\xi$ extends to an endomorphism on $N$, which clearly has to be an automorphism. That $\alpha_s\in \{\phi_u\}_{u\in U(e)}$ implies $\alpha_s$ and $\phi_k^{-1}$ commute, so
\[ \pi^{2d} = (\alpha_s\phi_{k}^{-1})^{2d} = \alpha_s^{2d}\phi_{k^{2d}}^{-1} =\mathrm{Id}_M.\]
Here $\alpha_s^2 = \mathrm{Id}_M$ by Lemma \ref{alpha lemma}(c) and $k^d\equiv 1\pmod{e}$ because $\mathrm{ord}_e(k)$ divides $d$. Since $\xi^2$ is clearly the identity on $P$, indeed $\xi$ has order dividing $2d$.

\vspace{1.5mm}

Next, we shall use induction on $\ell \in \mathbb{N}$ to show that
\begin{equation}\label{product}
(xyrs)\xi(xyrs)\cdots \xi^{\ell-1}(xyrs) = \begin{cases}
x^\ell y^\ell r^{\frac{\ell+1}{2}}s^{\ell} &\mbox{for $\ell$ odd},\\
x^\ell y^\ell r^{\frac{\ell}{2}}s^{\ell}&\mbox{for $\ell$ even}.
\end{cases}\end{equation}
The case $\ell=1$ is clear. For $\ell$ odd, observe that
\[ \xi^{\ell}(xyrs) = \pi^{\ell}(xy)\cdot r^{-1}\cdot rs = (\alpha_s^\ell\phi_{k^\ell}^{-1})(xy)s = (\alpha_s\phi_{k^\ell}^{-1})(xy)s.\]
Assuming that (\ref{product}) holds for $\ell$, we compute that
\begin{align*}
(xyrs)\xi(syrs)\cdots \xi^\ell(xyrs)& =x^{\ell}y^{\ell}r^{\frac{\ell+1}{2}}s^\ell \cdot (\alpha_s\phi_{k^\ell}^{-1})(xy)s\\
& = x^\ell y^\ell \cdot  (\alpha_r^{\frac{\ell+1}{2}}\alpha_s^{\ell+1}\phi_{k^{-\ell}})(xy)\cdot  r^{\frac{\ell+1}{2}}s^{\ell}\cdot s\\
& = x^\ell y^\ell \cdot x^{k^{-\ell}}y\cdot r^{\frac{\ell+1}{2}}s^{\ell+1} \hspace{1em} (\mbox{since }\alpha_r,\alpha_s^2=\mathrm{Id}_M)\\
& = x^{\ell+1}y^{\ell+1}r^{\frac{\ell+1}{2}} s^{\ell+1}
\end{align*}
and so (\ref{product}) also holds for $\ell+1$. Similarly, for $\ell$ even, observe that
\[ \xi^\ell(xyrs) = \pi^\ell(xy) \cdot r\cdot s  = (\alpha_s^{\ell} \phi_{k^\ell}^{-1})(xy) rs = \phi_{k^\ell}^{-1}(xy) rs.\]
Assuming that (\ref{product}) holds for $\ell$, we compute that
\begin{align*}
(xyrs)\xi(xyrs) \cdots \xi^{\ell}(xyrs) 
& = x^{\ell}y^{\ell}r^{\frac{\ell}{2}}s^\ell \cdot \phi_{k^\ell}^{-1}(xy)rs\\
& = x^{\ell}y^{\ell}\cdot (\alpha_r^{\frac{\ell}{2}}\alpha_s^\ell\phi_{k^{-\ell}})(xy)\cdot r^{\frac{\ell}{2}}s^\ell\cdot rs\\
& = x^\ell y^\ell \cdot x^{k^{-\ell}}y\cdot r^{\frac{\ell+2}{2}} s^{\ell+1}\hspace{1em} (\mbox{since }\alpha_r,\alpha_s^2=\mathrm{Id}_M)\\
& = x^{\ell+1} y^{\ell+1} r^{\frac{\ell+2}{2}} s^{\ell+1}.
\end{align*}
and so (\ref{product}) also holds for $\ell+1$. Hence, by induction, indeed we have (\ref{product}) for all $\ell\in \mathbb{N}$, and this immediately implies the second equality in (\ref{N set}).

\vspace{1.5mm}

To show the first equality in (\ref{N set}), since $N$ has order $n = 2^med$, it suffices to show that the set in (\ref{N set}) has at least $n$ elements. So suppose that
\begin{equation}\label{xi eqn} (xyrs)\xi(xyrs)\cdots \xi^{\ell_1-1}(xyrs) = (xyrs)\xi(xyrs)\cdots \xi^{\ell_2-1}(xyrs).\end{equation}
By (\ref{product}), this implies that $s^{\ell_1} \equiv s^{\ell_2}\pmod{\langle r\rangle}$ in the group $P$. But then $\ell_1,\ell_2$ have the same parity because $\langle s\rangle\cap \langle r\rangle = \langle s^2 \rangle$. Again by (\ref{product}), we have
\[ \begin{cases}
x^{\ell_1}y^{\ell_1}r^{\frac{\ell_1+1}{2}} s^{\ell_1} = x^{\ell_2}y^{\ell_2}r^{\frac{\ell_2+1}{2}} s^{\ell_2} &\mbox{for $\ell_1,\ell_2$ odd},\\
x^{\ell_1}y^{\ell_1}r^{\frac{\ell_1}{2}} s^{\ell_1} = x^{\ell_2}y^{\ell_2}r^{\frac{\ell_2}{2}} s^{\ell_2} &\mbox{for $\ell_1,\ell_2$ even}.
\end{cases}\]
Since $N = M\rtimes_\alpha P$ and $M = \langle x\rangle\rtimes \langle y\rangle$, in both cases, we see that $x^{\ell_1} = x^{\ell_2}$ \par\noindent  and $y^{\ell_1} = y^{\ell_2}$, which respectively imply that
\[ \ell_1 \equiv \ell_2\pmmod{e}\mbox{ and }\ell_1\equiv \ell_2\pmmod{d}.\]
In both cases, we also have $r^{\frac{\ell_1-\ell_2}{2}} = s^{\ell_2-\ell_1}$. Let us now prove that $s^{\ell_2-\ell_1} =1$ so in particular $r^{\frac{\ell_1-\ell_2}{2}}=1$. Note that $\ell_2 - \ell_1$ is always even.
\begin{itemize}
\item For $P = D_{2^m}$ with $m\geq 2$, since $s$ has order $2$, clearly $s^{\ell_2 - \ell_1}=1$.
\item For $P=Q_{2^m}$ with $m\geq 3$, since $s$ has order $4$, clearly $s^{\ell_2 - \ell_1}=1$ unless $\ell_2-\ell_1\equiv 2\pmod{4}$. So suppose  that $\ell_2-\ell_1\equiv2\pmod{4}$. Then
\[ r^{\frac{\ell_1-\ell_2}{2}} = s^{\ell_2-\ell_1-2}\cdot s^2 =  r^{2^{m-2}}\mbox{ and so } \frac{\ell_1-\ell_2}{2} \equiv 2^{m-2}\pmmod{2^{m-1}}.\]
But $m-1\geq 2$, so we obtain $\ell_1-\ell_2\equiv 0\pmod{4}$, which is a contradiction. This means that $\ell_2-\ell_1\equiv2\pmod{4}$ does not occur.
\end{itemize}
We have thus shown that $r^{\frac{\ell_1-\ell_2}{2}}=1$, which implies
\[ \frac{\ell_1-\ell_2}{2} \equiv 0\pmmod{2^{m-1}}\mbox{ and thus }\ell_1\equiv \ell_2 \pmmod{2^m}.\]
Since $2^m,e,d$ are pairwise coprime, we deduce that $\ell_1\equiv\ell_2\pmod{n}$. Therefore, indeed the set in (\ref{N set}) contains at least $n$ distinct elements.
\end{proof}

\section{Proof of Theorem \ref{main thm}}

Let $N$ be a non-$C$-group of order $n$. By Theorem \ref{pre thm}, we may assume that 
\[ N = M \rtimes_\alpha P \mbox{ with }\alpha\in\Hom(P,\Aut(M)),\]
where $M$ is a $C$-group of odd order, and $P$ is either $D_{2^m}$ with $m\geq 2$ or $Q_{2^m}$ with $m\geq 3$.  We wish to show that $(C_n,N)$ is realizable if and only if
\begin{equation}\label{reduction}
\begin{cases}
\alpha(P)\mbox{ has order $1$ or $2$} &\mbox{when }P=D_4\mbox{ or }P=Q_8,\\
\alpha(r)=\mathrm{Id}_M&\mbox{otherwise}.
\end{cases} 
\end{equation}
The main ingredients are Propositions \ref{P' prop}, \ref{iso prop}, and \ref{exists prop}.


First, suppose that $(C_n,N)$ is realizable. By Proposition \ref{regular lem}, this implies that there exist $\ff\in\Hom(C_n,\Aut(N))$ and a bijective $\fg\in Z_\ff^1(C_n,N)$. Notice that $M_0=M\rtimes_\alpha P'$ is a characteristic subgroup of $N$ and put $H=\fg^{-1}(M_0)$, which is a subgroup of $C_n$ by Proposition \ref{char prop}. Trivially $H$ lies in the center of $C_n$, so by the proof of Proposition \ref{char prop}(b), the homomorphism
\[ \overline{\ff}_{M_0}\in\Hom(C_n/H,\Aut(N/M_0));\hspace{1em}
\overline{\ff}_{M_0}(\sigma H) = (\eta M_0 \mapsto \ff(\sigma)(\eta) M_0)\]
and the bijective crossed homomorphism
\[\overline{\fg}_{M_0}\in Z^1_{\overline{\ff}_{M_0}}(C_n/H,N/M_0);\hspace{1em}
\overline{\fg}_{M_0}(\sigma H) = \fg(\sigma)M_0\]
are well-defined. Note that $\overline{\ff}_{M_0}$ cannot be trivial, for otherwise $\overline{\fg}_{M_0}$ would be an isomorphism by (\ref{g relation}), which cannot happen because $N/M_0\simeq P/P'\simeq D_4$ by Lemma \ref{commutator lem} while $C_n/H$ is cyclic.

\vspace{1.5mm}

Now, assume for contradiction that (\ref{reduction}) does not hold. This implies that $\ker(\alpha) = P'$ when $P = D_4$ or $P=Q_8$ in view of (\ref{ker poss}), and that $\alpha(r)\neq\mathrm{Id}_M$ otherwise. It follows from Proposition \ref{P' prop} that the canonical homomorphism
\[\begin{tikzcd}[column sep = 3.5cm]
\Aut(N)\arrow{r}{\xi\mapsto (\eta M_0\mapsto \xi(\eta)M_0)} &\Aut(N/M_0)\arrow[equal]{r}{\text{identification}} & \Aut(P/P').
\end{tikzcd}\]
is trivial. But then $\overline{\ff}_{M_0}$ would be trivial, which we know is impossible. This implies that (\ref{reduction}) must hold, as desired.

\vspace{1.5mm}


Conversely, assume that (\ref{reduction}) holds. By Proposition \ref{iso prop}, we may modify $\alpha$ if necessary so that the hypothesis of Proposition \ref{exists prop} is satisfied. We then deduce that there exist $\xi\in \Aut(N)$ and $\eta_0\in N$ such that
\begin{enumerate}[(i)]
\item $\xi^{n} = \mathrm{Id}_N$ and $\eta_0\xi(\eta_0)\cdots \xi(\eta_0)^{n-1}=1$;
\item $N = \{\eta_0\xi(\eta_0)\cdots \xi^{\ell-1}(\eta_0):\ell\in\mathbb{N}\}$.
\end{enumerate}
Consider $\rho(\eta_0)\xi$, which is an element of $\Hol(N)$. For any $\ell\in\mathbb{N}$, we have
\[ (\rho(\eta_0)\xi)^\ell= \rho(\eta_0\xi(\eta_0)\cdots \xi^{\ell-1}(\eta_0))\cdot\xi^\ell.\]
Then $\rho(\eta_0)\xi $ has order dividing $n$ by (i), and $\langle\rho(\eta_0)\xi\rangle$ acts transitively on $N$ by (ii). It follows that $\langle \rho(\eta_0)\xi\rangle$ is in fact a regular subgroup of $\Hol(N)$ whose order is exactly $n$. This proves that  $(C_n,N)$ is realizable.

\section*{Acknowledgments} 

This work was supported by JSPS KAKENHI (Grant-in-Aid for Research Activity Start-up) Grant Number 21K20319. 
The author thanks the referee for suggesting the discussion in Remark \ref{remark Rump} and the proof in Remark \ref{pre remark}.


\begin{thebibliography}{99}

\bibitem{AB}
A. A. Alabdali and N. P. Byott, \emph{Counting Hopf-Galois structures on cyclic field extensions of squarefree degree}, J. Algebra 493 (2018), 1--19.

\bibitem{AB'}
A. A. Alabdali and and N. P. Byott, \emph{Hopf-Galois structures of squarefree degree}, J. Algebra 559 (2020), 58--86.

\bibitem{AB''}
A. A. Alabdali and and N. P. Byott, \emph{Skew braces of squarefree order}, J. Algebra Appl. 20 (2021), no. 7, Paper No. 2150128, 21 pp. 

\bibitem{Bachiller}
D. Bachiller, \emph{Counterexample to a conjecture about braces}, J. Algebra 453 (2016), 160--176.

\bibitem{magma}
W. Bosma, J. Cannon, and C. Playoust, \emph{The Magma algebra system. I. The user language}, J. Symbolic Comput., 24 (1997), 235--265.

\bibitem{By96}
N. P. Byott, \emph{Uniqueness of Hopf Galois structure for separable field extensions}, Comm. Algebra 24 (1996), no. 10, 3217--3228. Corrigendum, \emph{ibid.} no. 11, 3705.

\bibitem{Byott 2}
N. P. Byott, \emph{Hopf-Galois structures on almost cyclic field extensions of $2$-power degree}, J. Algebra 318 (2007), 351--371.

\bibitem{Byott nilpotent}
N. P. Byott, \emph{Nilpotent and abelian Hopf-Galois structures on field extensions}, J. Algebra 381 (2013), 131--139.

\bibitem{Byott simple}
N. P. Byott, \emph{Hopf-Galois structures on field extensions with simple Galois groups}, Bull. London Math.
Soc. 36 (2004), no. 1, 23--29.

\bibitem{Byott soluble}
N. P. Byott, \emph{Solubility criteria for Hopf-Galois structures}, New York J. Math. 21 (2015), 883--903.

\bibitem{BC}
N. P. Byott and L. N. Childs, \emph{Fixed-point free pairs of homomorphisms and nonabelian Hopf-Galois structures}, New York J. Math. 18 (2012), 707--731.

\bibitem{Childs book}
L. N. Childs, \emph{Taming wild extensions: Hopf algebras and local Galois module theory}. Mathematical Surveys and Monographs, 80. American Mathematical Society, Providence, RI, 2000.

\bibitem{biskew}
L. N. Childs, \emph{Bi-skew braces and Hopf Galois structures}, New York J. Math. 25 (2019), 574--588.

\bibitem{Conrad D}
K. Conrad, \emph{Dihedral groups II}, online notes, retrieved on December 16, 2021.
\\\url{https://kconrad.math.uconn.edu/blurbs/grouptheory/dihedral2.pdf}
\bibitem{Conrad Q}
K. Conrad, \emph{Generalized quaternions}, online notes, retrieved on December 16, 2021.
\\\url{https://kconrad.math.uconn.edu/blurbs/grouptheory/genquat.pdf}


\bibitem{FCC}
S. C. Featherstonhaugh, A. Caranti, and L. N. Childs, \emph{Abelian Hopf Galois structures on prime-power Galois field extensions}, Trans. Amer. Math. Soc. 364 (2012), no. 7, 3675--3684.

\bibitem{Skew braces}
L. Guarnieri and L. Vendramin, \emph{Skew braces and the Yang-Baxter equation}, Math. Comp. 86 (2017), no. 307, 2519--2534.

\bibitem{Kohl}
T. Kohl, \emph{Classification of the Hopf Galois structures on prime power radical extensions}, J. Algebra 207 (1998), 525--546.



\bibitem{MM}
M. Ram Murty and V. Kumar Murty, \emph{On groups of squarefree order}, Math. Ann. 267 (1984), 299--309.

\bibitem{Rump}
W. Rump, \emph{Classification of cyclic braces, II.}, Trans. Amer. Math. Soc. 372 (2019), no. 1, 305--328.

\bibitem{Nasy}
T. Nasybullov, \emph{Connections between properties of the additive and the multiplicative groups of a two-sided skew brace}, J. Algebra 540 (2019), 156--167.

\bibitem{NZ}
K. Nejabati Zenouz, \emph{On Hopf-Galois structures and skew braces of order $p^3$}, Ph.D. thesis, University of Exeter (2018).

\bibitem{GT book}
E. Schenkman, \emph{Group theory}. D. Van Nostrand Co., Inc., Princeton, N.J.-Toronto, Ont.-London 1965.


\bibitem{Tsang HG}
C. Tsang, \emph{Non-existence of Hopf-Galois structures and bijective crossed homomorphisms}, J. Pure Appl. Algebra 223 (2019), no. 7, 2804--2821. 

\bibitem{Tsang Sn}
C. Tsang, \emph{Hopf-Galois structures on a Galois $S_n$-extension}. J. Algebra 531 (2019), 349--360.

\bibitem{Tsang solvable}
C. Tsang and C. Qin, \emph{On the solvability of regular subgroups in the holomorph of a finite solvable group}, Internat. J. Algebra Comput. 30 (2020), no. 2, 253--265.  

\bibitem{Tsang ASG}
C. Tsang, \emph{Hopf-Galois structures on finite extensions with almost simple Galois group}. J. Number Theory 214 (2020), 286--311.

\bibitem{Tsang QS}
C. Tsang, \emph{Hopf-Galois structures on finite extensions with quasisimple Galois group}, Bull. London Math. Soc. 53 (2021), no. 1, 148--160.

\end{thebibliography}
\end{document}